\documentclass[]{article}
\usepackage{graphicx}
\usepackage{amsmath}
\usepackage{amssymb}
\usepackage{amsthm}
\usepackage{enumerate}
\usepackage{color}
\usepackage{mathdots}
\usepackage{sectsty}
\usepackage[hidelinks]{hyperref}
\usepackage{tikz}
\usepackage{caption}
\usepackage{adjustbox}
\usepackage{fancyhdr}
\usepackage{verbatim}
\usepackage{url}

\sectionfont{\scshape\centering\fontsize{11}{14}\selectfont}
\subsectionfont{\scshape\fontsize{11}{14}\selectfont}

\newcommand\shorttitle{Restricted Bipartite Partitions}
\newcommand\authors{\small Nian Hong Zhou}

\hypersetup{
    colorlinks=true,       
    linkcolor=blue,          
    citecolor=cyan,        
}

\theoremstyle{plain}
\newtheorem{theorem}{Theorem}[section]
\newtheorem{lemma}[theorem]{Lemma}

\newtheorem{proposition}[theorem]{Proposition}

\theoremstyle{remark}
\newtheorem{remark}{Remark}[section]

\makeatletter

\newcommand{\Rmnum}[1]{\expandafter\@slowromancap\romannumeral #1@}

\def\rb{\mathbb R}
\def\nb{\mathbb N}
\def\zb{\mathbb Z}
\def\cb{{\mathbb C}}

\def\rrw{\rightarrow}

\numberwithin{equation}{section}

\title{\large \bf UNIFORM ASYMPTOTIC FORMULAS OF RESTRICTED BIPARTITE PARTITIONS}
\author{\small NIAN HONG ZHOU}

\date{}

\begin{document}

\maketitle




\begin{abstract}
In this paper, we investigate $\pi(m,n)$, the number of partitions of the \emph{bipartite number} $(m,n)$ into \emph{steadily
decreasing} parts, introduced by L.Carlitz ['A problem in partitions',
Duke Math Journal 30 (1963), 203--213]. We give a relation between $\pi(m,n)$ and the crank statistic $M(m,n)$ for integer partitions.
Using this relation, some uniform asymptotic formulas for $\pi(m,n)$ are established.
\end{abstract}


\section{Introduction and statement of results}
We begin with some standard definitions from the theory of partitions \cite{MR1634067}.
An integer \emph{partition} is a non-increasing sequence $\lambda_1, \lambda_2,\dots,$  such that each $\lambda_j$ is a nonnegative integer.
The partition $(\lambda_1,\lambda_2, \dots)$ will be denoted by $\lambda$.  We say $\lambda$ is a partition of $n$
if $\lambda_1+\lambda_2+\dots=n$.  Let $p(n)$ be the number of partitions of $n$ and let $p(0):=1$. Then by Euler, we have the following famous generating function
\begin{equation}\label{eqp}
\sum_{n\ge 0}p(n)q^n=\frac{1}{(q;q)_{\infty}},~(q\in\cb, |q|<1).
\end{equation}
Here $(a;q)_{\infty}=\prod_{j\ge 0}(1-aq^j)$ for any $a\in\cb$ and $|q|<1$.
One of the most celebrated result of the integer partition is the Hardy--Ramanujan asymptotic formula:
\begin{equation}\label{eqap}
p(n)\sim \frac{1}{4\sqrt{3}n}e^{2\pi\sqrt{n/6}},
\end{equation}
as integer $n\rrw +\infty$, see \cite{MR1575586}.

For partitions $\alpha=(\alpha_1,\alpha_2, \dots)$ and $\beta=(\beta_1,\beta_2, \dots)$, follows from \cite[p.207]{MR1634067} we say that the pair $(\alpha,\beta)$
is a pair of partitions with steadily decreasing parts if
$$\min(\alpha_{i}, \beta_{i})\ge \max(\alpha_{i+1}, \beta_{i+1}),$$
holds for all integers $i\ge 1$. Let $\pi(m,n)$ be the number of partitions of the \emph{bipartite number} $(m,n)$ of the form
$$(m,n)=(\alpha_1+\alpha_2+\dots, ~\beta_1+\beta_2+\dots),$$
with each pair $(\alpha,\beta)$ has steadily decreasing parts. A generating function for $\pi(m,n)$ is given by Carlitz \cite{MR148636, MR0175796}
\begin{equation}\label{eqp2}
\sum_{m,n\ge 0}\pi(m,n)x^{m}y^n=\frac{1}{(x;xy)_{\infty}(x^2y^2; x^2y^2)_{\infty}(y;xy)_{\infty}},
\end{equation}
for all $x, y\in\cb$ with $|x|, |y|<1$. This is analogous to the generating function \eqref{eqp} for the number of partitions of \emph{$1$-partite number}.
In \cite{MR437350}, Andrews extended \eqref{eqp2} to \emph{$r$-partite number} for any positive integer $r$. For more related results,
see \cite{MR190118, MR202689, MR1911463}.

In this paper, we investigate the asymptotics of $\pi(m,n)$ analogous to the Hardy--Ramanujan asymptotic formula \eqref{eqap}.
To state our main results, we need the cubic partition function $c(n)$ introduced by Chan \cite{MR2652901} that
\begin{equation}\label{eqg1}
\sum_{n\ge 0}c(n)q^n=\frac{1}{(q;q)_{\infty}(q^2;q^2)_{\infty}};
\end{equation}
and the crank statistic for integer partitions, introduced and investigated by Dyson \cite{MR3077150} and
Andrews and Garvan \cite{MR920146, MR929094}. Denoting by $M(m, n)$ the number of partitions of $n$ with crank $m$, we
have the generating functions
\begin{equation}\label{eqg2}
\sum_{\substack{n\ge 0 \\ m\in\zb}}M(m,n)q^n\zeta^m=\frac{(q;q)_{\infty}}{(\zeta q;q)_{\infty}(\zeta^{-1}q;q)_{\infty}}=\frac{1-\zeta}{(q;q)_{\infty}}\sum_{n\in\zb} \frac{(-1)^nq^{\frac{n(n+1)}2}}{1-\zeta q^n}.
\end{equation}

The first result of this paper is stated as follows.

\begin{proposition}\label{pr1}Let $m$ and $n$ be non-negative integers. We have
\begin{equation}\label{eq11}
\pi(m,n)=\sum_{0\le k\le \min(m,n)}c\left(\min(m,n)-k\right)\alpha(|m-n|,k),
\end{equation}
where
$$\alpha(s,k)=\sum_{\ell\ge 0}(-1)^{\ell}p\left(k-\ell(\ell+1)/{2}-\ell s\right),$$
with $p(r):=0$ for all $r<0$. In particular, for each integer $k\ge 0$, $\pi(0,k)=\pi(k,0)=1$.
Let $D(m,n):=\pi(m,n)-\pi(m-1, n)$ with $\pi(-1,n):=0$. Then, we have
\begin{equation}\label{eq12}
D(m,n)=\sum_{0\le k\le L_{m,n}}c(L_{m,n}-k)M(n-L_{m,n},n-L_{m,n}+k),
\end{equation}
where $L_{m,n}:=\min(2n-m,m)$.
In particular, if $m>2n$ then $D(m,n)=0$.
\end{proposition}

By use of \eqref{eq12} of Proposition \ref{pr1}, we prove the following uniform asymptotic behavior for $D(m,n)$, by using some results on the uniform
asymptotics of $M(m,n)$, proved by the author in \cite{MR3924736}.
\begin{theorem}\label{th1}Uniformly for all integers $m,n>0$ such that $m\le 2n$,
\begin{equation*}
D(m,n)\sim \frac{5c}{2^5\cdot 3}\frac{e^{c\sqrt{\min(m,2n-m)}}}{[\min(m,2n-m)]^2}\left(1+e^{-\frac{c|n-m|}{2\sqrt{\min(m,2n-m)}}}\right)^{-2},
\end{equation*}
as $\min(m,2n-m)\rrw \infty$, where $c=2\pi\sqrt{5/12}$.
\end{theorem}
Aa a consequence of the above theorem, we prove the following asymptotic formula for $\pi(m,n)$ which analogous the Hardy--Ramanujan asymptotic formula \eqref{eqap}.
\begin{theorem}\label{th2}Uniformly for all integers $m,n>0$,
\begin{equation*}
\pi(m,n)\sim \frac{5}{2^4\cdot 3}\frac{e^{c\sqrt{\min(m,n)}}}{[\min(m,n)]^{3/2}}\left(1+e^{-\frac{c|n-m|}{2\sqrt{\min(m,n)}}}\right)^{-1},
\end{equation*}
as $\min(m,n)\rrw \infty$, where $c=2\pi\sqrt{5/12}$. In particular,
\begin{equation*}
\pi(n,n)\sim \frac{5}{2^5\cdot 3}\frac{e^{c\sqrt{n}}}{n^{3/2}},
\end{equation*}
 as $n\rrw \infty$.
\end{theorem}

\begin{remark}\emph{
Since $\alpha(s,k)$ of Proposition \ref{pr1} has a similar expression to $V\left(\ell, N+\frac{|\ell|^2+|\ell|}{2}\right)$ the number of strongly concave compositions of
$N+\frac{|\ell|^2+|\ell|}{2}\in\nb$ with rank $\ell\in\zb$, of \cite[Proposition 1.2]{Zhou2019} proved by the author, it is possible to give a proof of Theorem \ref{th2}, by using the method used in the proof of \cite[Theorem 1.3]{Zhou2019}.
}
\end{remark}
By using \eqref{eq11} of Proposition \ref{pr1}, we illustrate some of our results in the following(All computations are done in {\bf Mathematica}).
\begin{center}
\captionof{table}{Numerical data for $\pi(m,n)$.}
  \begin{tabular}{ c | c | c | c  }
\hline
    $L$ & $\pi(L^2,L^2)$ & $A(L^2,L^2)$ & $\frac{\pi(L^2,L^2)}{A(L^2,L^2)}$\\
\hline
   $10$ & $2.02082\cdot 10^{13}$ & $2.14152\cdot 10^{13}$ & $\sim 0.9436$ \\
   $40$ & $2.29293\cdot 10^{64}$ & $2.32601\cdot 10^{64}$ & $\sim 0.9858$ \\
   $70$ & $2.99238\cdot 10^{116}$ & $3.01693\cdot 10^{116}$ & $\sim 0.9919$ \\
   $100$ & $7.15231\cdot 10^{168}$ & $7.19331 \cdot 10^{168}$ & $\sim 0.9943$ \\
\hline

\hline
    $L$ & $\pi(L^2,L^2+L)$ & $A(L^2,L^2+L)$ & $\frac{\pi(L^2,L^2+L)}{A(L^2,L^2+L)}$\\
\hline
   $10$ & $3.42924\cdot 10^{13}$ & $3.78489\cdot 10^{13}$ & $\sim 0.9060$ \\
   $40$ & $4.00991\cdot 10^{64}$ & $4.11096\cdot 10^{64}$ & $\sim 0.9754$ \\
   $70$ & $5.25671\cdot 10^{116}$ & $5.33209\cdot 10^{116}$ & $\sim 0.9859$ \\
   $100$ & $1.25872\cdot 10^{169}$ & $1.27134 \cdot 10^{169}$ & $\sim 0.9901$ \\
\hline
  \end{tabular}
\end{center}
~Here $A(m,n)=\frac{5}{2^4\cdot 3}\frac{e^{2\pi\sqrt{5m/12}}}{m^{3/2}}\left(1+e^{-\frac{\pi\sqrt{5/12}(n-m)}{\sqrt{m}}}\right)^{-1}$.

\paragraph{Acknowledgements.}The author would like to thank the anonymous referees for their very helpful
comments and suggestions. This research was supported by the National Science Foundation of China (Grant No. 11971173).

\section{Proofs of results}
\subsection{The proof of Proposition \ref{pr1}}~

Setting $q=xy$ and $\zeta=x$, the generating function \eqref{eqp2} can be rewritten as
\begin{align}\label{eqg11}
\sum_{m,n\ge 0}\pi(m,n)q^{n}\zeta^{m-n}&=\frac{1}{(q; q)_{\infty}(q^2; q^2)_{\infty}}\frac{(q;q)_{\infty}}{(\zeta; q)_{\infty}(\zeta^{-1}q; q)_{\infty}}\\
\label{eqg22}
&=\frac{1}{(q; q)_{\infty}^2(q^2; q^2)_{\infty}}\sum_{n\in\zb} \frac{(-1)^nq^{\frac{n(n+1)}2}}{1-\zeta q^n},
\end{align}
by using \eqref{eqg2}. Therefore, by use of \eqref{eqg22}, we have for each $m\ge 0$,
\begin{align*}
\sum_{n\ge 0}\pi(m+n,n)q^{n}&=\frac{1}{(q; q)_{\infty}(q^2; q^2)_{\infty}}\frac{1}{(q;q)_{\infty}}\sum_{n\ge 0}(-1)^{n}q^{\frac{n(n+1)}{2}+nm}\\
&=\sum_{s\ge 0}c(s)q^s\sum_{n\ge 0}\left(\sum_{\ell\ge 0}(-1)^{\ell}p(n-\ell(\ell+1)/2-m\ell)\right)q^n\\
&=\sum_{n\ge 0}\left(\sum_{0\le k\le n}c(n-k)\alpha(m,k)\right)q^n.
\end{align*}
That is if $m\ge n$ then
\begin{align*}
\pi(m,n)=\sum_{0\le k\le n}c(n-k)\alpha(m-n,k).
\end{align*}
From \eqref{eqp2} we observe that $\pi(m,n)=\pi(n,m)$, and the proof of \eqref{eq11} follows.  We now proof \eqref{eq12}. By noting that $\pi(-1, n):=0$ for all integers $n\ge 0$, and using \eqref{eqg2} and \eqref{eqg11} implies that
\begin{align*}
\sum_{\substack{n\ge 0\\ m\ge 0}}(\pi(m,n)-\pi(m-1,n))q^{n}\zeta^{m-n}&=\frac{1}{(q; q)_{\infty}(q^2; q^2)_{\infty}}\sum_{\substack{n\ge 0\\ m\in\zb}}M(m,n)q^n\zeta^{m}.
\end{align*}
Using \eqref{eqg1} we further obtain that
\begin{align*}
D(m,n)=\sum_{0\le \ell\le n}c(n-\ell)M(m-n,\ell).
\end{align*}
Recall the well known results that $M(m,n)=M(-m,n)$, and $M(m,n)=0$ if $|m|>n$, we have:
\begin{align}\label{eqq1}
D(m,n)&=\sum_{0\le \ell\le n}c(n-\ell)M(n-m,\ell)\nonumber\\
&=\sum_{0\le k\le m}c(m-k)M(n-m,n-m+k).
\end{align}
holds for $0\le m\le n$,
\begin{align}\label{eqq2}
D(m,n)&=\sum_{0\le \ell\le n}c(n-\ell)M(m-n,\ell)\nonumber\\
&=\sum_{0\le k\le 2n-m}c(2n-m-k)M(n-(2n-m),n-(2n-m)+k).
\end{align}
holds for  $n\le m\le 2n$, and
\begin{equation}\label{eqq3}
D(m,n)=\sum_{m-n\le \ell\le n}c(n-\ell)M(m-n,\ell)=0.
\end{equation}
holds for $m-n>n$, that is $m>2n$. Combining \eqref{eqq1}--\eqref{eqq3} we get the proof of \eqref{eq12}.

\subsection{Auxiliary lemmas}~

To prove Theorem \ref{th1}, we need the following uniform asymptotics of $M(m,n)$, which follows from \cite[Corollary 1.4]{MR3924736}.
We note that the uniform asymptotics of $M(m,n)$ was first considered by Dyson \cite{MR1001259} as an open problem,
proved first by Bringmann and Dousse \cite{MR3451872}, and completed as the following form by the author \cite{MR3924736}.

\begin{proposition}\label{pro1}
Uniformly for all integers $\ell, k\ge 0$, as $\ell\rrw\infty$,
$$M(k,k+\ell)\sim \frac{\pi }{12\sqrt{2}}\left(1+e^{-\frac{\pi k}{\sqrt{6\ell}}}\right)^{-2}\frac{e^{2\pi\sqrt{\ell/6}}}{\ell^{3/2}}.$$
\end{proposition}
\begin{proof}
From \cite[Corollary 1.4]{MR3924736} and the Hardy--Ramanujan asymptotic formula \eqref{eqap}, we have as $\ell\rrw \infty$,
\begin{align*}
M(k,k+\ell)&\sim \frac{\pi }{\sqrt{6}}\left(1+e^{-\frac{\pi k}{\sqrt{6(\ell+k)}}}\right)^{-2}\frac{p(\ell)}{\ell^{3/2}}\\
&\sim \frac{\pi}{12\sqrt{2}}\frac{e^{2\pi\sqrt{\ell/6}}}{\ell^{3/2}}\left(1+e^{-\frac{\pi k}{\sqrt{6(\ell+k)}}}\right)^{-2}\\
&\sim \frac{\pi}{12\sqrt{2}}\frac{e^{2\pi\sqrt{\ell/6}}}{\ell^{3/2}}\left(1+{\bf 1}_{\ell> k^{2-1/8}}e^{-\frac{\pi k}{\sqrt{6(\ell+k)}}}\right)^{-2}.
\end{align*}
Here and throughout, ${\bf 1}_{condition} = 1$ if the 'condition' is true, and equals to $0$ if the 'condition' is false.
Notice that if $\ell> k^{2-1/8}$ and $\ell\rrw +\infty$ then
$$\frac{\pi k}{\sqrt{6(\ell+k)}}=\frac{\pi k}{\sqrt{6\ell}}(1+O(\ell^{-1}k))=\frac{\pi k}{\sqrt{6\ell}}+O\left(\ell^{-\frac{1}{2}+\frac{1}{15}}\right),$$
we have
\begin{align*}
M(k,k+\ell)
&\sim \frac{\pi}{12\sqrt{2}}\frac{e^{2\pi\sqrt{\ell/6}}}{\ell^{3/2}}\left(1+{\bf 1}_{\ell> k^{2-1/8}}e^{-\frac{\pi k}{\sqrt{6\ell}}+O\left(\ell^{-\frac{1}{2}+\frac{1}{15}}\right)}\right)^{-2}\\
&=\frac{\pi}{12\sqrt{2}}\frac{e^{2\pi\sqrt{\ell/6}}}{\ell^{3/2}}\left(1+{\bf 1}_{\ell> k^{2-1/8}}e^{-\frac{\pi k}{\sqrt{6\ell}}}\right)^{-2}\left(1+O\left(\ell^{-\frac{1}{2}+\frac{1}{15}}\right)\right)\\
&\sim \frac{\pi}{12\sqrt{2}}\frac{e^{2\pi\sqrt{\ell/6}}}{\ell^{3/2}}\left(1+e^{-\frac{\pi k}{\sqrt{6\ell}}}\right)^{-2},
\end{align*}
which completes the proof.
\end{proof}

We also need the asymptotics of the cubic partitions $c(n)$, which can be find in \cite[Equation (1.5)]{MR3514335}.
\begin{lemma}\label{cblem} We have
$$c(n)\sim \frac{1}{8n^{5/4}}e^{\pi\sqrt{n}},$$
as integer $n\rrw +\infty$.
\end{lemma}
We finally need
\begin{lemma}\label{flem}Define for all $x\in[0, 1]$ that
$$f(x)=\sqrt{1-x}+\sqrt{2x/3}.$$
Then $f(x)$ is increasing on $[0, 2/5]$ and decreasing on $[2/5, 1]$. Moreover,
$$f(2/5+t)=\sqrt{{5}/{3}}-\kappa t^{2}+O(|t|^{3}),$$
as $t\rrw 0$, where $\kappa:=2^{-4}\cdot 3^{-3/2}\cdot 5^{5/2}$.
\end{lemma}
\begin{proof}The proof of this lemma is direct and we shall omit it.
\end{proof}
\subsection{The proof of Theorem \ref{th1} and Theorem \ref{th2}}~

In this subsection, we always assume that $m,n$ are integers with $n\ge m>0$ and $m\rrw \infty$.

We first prove Theorem \ref{th1}. From Proposition \ref{pr1}, we split that
\begin{align*}
D(m,n)&=\sum_{0\le k\le m}c(m-k)M(n-m,n-m+k)\\
&=\left(\sum_{\left|k-\frac{2}{5}m\right|\le m^{\frac{3}{4}+\frac{1}{16}}}
+\sum_{\substack{0\le k\le m\\ |k-2m/5|> m^{3/4+2^{-4}}}}\right) c(m-k)M(n-m,n-m+k)\\
&=:I(m,n)+E(m,n).
\end{align*}
For $E(m,n)$ defined as above, using Proposition \ref{pro1} and Lemma \ref{cblem} we have:
\begin{align*}
E(m,n)&= c(m)+M(n-m,n)+\sum_{\substack{1\le k< m\\ |k-2m/5|> m^{3/4+2^{-4}}}}c(m-k)M(n-m,n-m+k)\\
&\ll \frac{e^{\pi\sqrt{m}}}{m}+\frac{e^{2\pi\sqrt{m/6}}}{m^{3/2}}+\sum_{\substack{1\le k< m\\ |k-2m/5|> m^{3/4+2^{-4}}}}\frac{e^{\pi (\sqrt{m-k}+\sqrt{2k/3})}}{k^{3/2}(m-k)}\\
&\ll e^{\pi\sqrt{m}}+\sum_{\substack{1\le k< m\\ |k-2m/5|> m^{3/4+2^{-4}}}}\frac{1}{k^{3/2}}e^{\pi \sqrt{m}f(k/m)}.
\end{align*}
By use of Lemma \ref{flem}, we further find that
\begin{align}\label{ee}
E(m,n)
&\ll  e^{\pi\sqrt{m}}+ e^{\pi \sqrt{m}f\left(2/5+m^{-1/4+2^{-4}}\right)}+e^{\pi \sqrt{m}f\left(2/5-m^{-1/4+2^{-4}}\right)}\ll e^{\pi\sqrt{5m/3}-\kappa\pi m^{1/8}}.
\end{align}
We now evaluate $I(m,n)$. The using of Proposition \ref{pro1} and Lemma \ref{cblem} implies that
\begin{align*}
I(m,n)&\sim \frac{\pi }{96\sqrt{2}}\sum_{\left|k-\frac{2}{5}m\right|\le m^{\frac{3}{4}+\frac{1}{16}}}\left(1+e^{-\frac{\pi (n-m)}{\sqrt{6k}}}\right)^{-2}
\frac{e^{\pi\sqrt{(m-k)}}}{(m-k)^{5/4}}\frac{e^{2\pi\sqrt{k/6}}}{k^{3/2}}\\
&\sim \frac{\pi }{96\sqrt{2}(3m/5)^{5/4}(2m/5)^{3/2}}\sum_{\left|k-\frac{2}{5}m\right|\le m^{\frac{3}{4}+\frac{1}{16}}}
\frac{e^{\pi\sqrt{m} f(k/m)}}{\left(1+e^{-\frac{\pi (n-m)}{\sqrt{6k}}}\right)^{2}}.
\end{align*}
By use of Lemma \ref{flem} we further obtain that
\begin{align}\label{impe}
I(m,n)\sim  \frac{\pi \left(1+e^{-(1+O(m^{-3/16}))\frac{\sqrt{5}\pi (n-m)}{\sqrt{12m}}}\right)^{-2}e^{\pi\sqrt{5m/3}}}{96\sqrt{2}(3m/5)^{5/4}(2m/5)^{3/2}}
\sum_{\left|k-\frac{2}{5}m\right|\le m^{\frac{3}{4}+\frac{1}{16}}}e^{-\frac{\kappa \pi}{ m^{3/2}}(k-2m/5)^2}.
\end{align}
Since $n\ge m$ and $m\rrw +\infty$, we have
\begin{align}\label{fpe}
\left(1+e^{-(1+O(m^{-3/16}))\frac{\sqrt{5}\pi (n-m)}{\sqrt{12m}}}\right)^{-2}&\sim \left(1+{\bf 1}_{m>(n-m)^{2-1/8}}e^{-(1+O(m^{-3/16}))
\frac{\sqrt{5}\pi (n-m)}{\sqrt{12m}}}\right)^{-2}\nonumber\\
&= \left(1+{\bf 1}_{m>(n-m)^{2-1/8}}e^{-\pi\sqrt{\frac{5}{12m}}(n-m)+O\left(m^{\frac{1}{30}-\frac{3}{16}}\right)}\right)^{-2}\nonumber\\
&\sim  \left(1+e^{-\pi\sqrt{\frac{5}{12m}}(n-m)}\right)^{-2}.
\end{align}
By using Abel's summation formula, it is easy to find that
\begin{equation}\label{ipe}
\sum_{|k-2m/5|\le m^{3/4+2^{-4}}}e^{-\pi \kappa m^{-3/2}(k-2m/5)^2}\sim \int_{\rb}e^{-\pi \kappa m^{-3/2}x^2}\,dx=\frac{m^{3/4}}{\sqrt{\kappa}},
\end{equation}
as $m\rrw +\infty$. Substituting \eqref{fpe} and \eqref{ipe} to \eqref{impe}, and note that $\kappa=2^{-4}\cdot 3^{-3/2}\cdot 5^{5/2}$ we further obtain that
\begin{align*}
I(m,n)&\sim \frac{\pi m^{3/4}e^{2\pi\sqrt{\frac{5m}{12}}}}{96\sqrt{2}(3m/5)^{5/4}(2m/5)^{3/2}\kappa^{1/2}}\left(1+e^{-\pi\sqrt{\frac{5}{12m}}(n-m)}\right)^{-2}\\
&=\frac{5\cdot \pi}{2^4\cdot 3}\sqrt{\frac{5}{12}}\frac{e^{2\pi\sqrt{\frac{5m}{12}}}}{m^2}\left(1+e^{-\pi\sqrt{\frac{5}{12m}}(n-m)}\right)^{-2}.
\end{align*}
Therefore by Combining \eqref{ee} we find that
\begin{equation*}
D(m,n)\sim \frac{5c}{2^5\cdot 3}\frac{e^{c\sqrt{m}}}{m^2}\left(1+e^{-\frac{c(n-m)}{2\sqrt{m}}}\right)^{-2},
\end{equation*}
with $c=2\pi\sqrt{5/12}$, holds for $m\le n$ and $m\rrw +\infty$. Using \eqref{eq12} then the proof of Theorem \ref{th1} follows.\newline

We now prove Theorem \ref{th2}. Since $D(m,n)=\pi(m,n)-\pi(m-1,n)$ and for all integers $k,n\ge 1$ such that $k\le n$,
$$D(k,n)\ll k^{-2}e^{c\sqrt{k}},$$
by using Theorem \ref{th1}, we have
\begin{align*}
\pi(m,n)&=\pi(0,n)+\sum_{1\le k\le m}D(k,n)\ll 1+\sum_{1\le k\le m}k^{-2}e^{c\sqrt{k}}\ll e^{c\sqrt{m}}.
\end{align*}
Let $\lfloor\cdot\rfloor$ be the greatest integer function. Using Theorem \ref{th1} again,
\begin{align*}
\pi(m,n)&=\pi\left(m-\lfloor m^{{9}/{16}}\rfloor, n\right)+\sum_{m-\lfloor m^{{9}/{16}}\rfloor<k\le m}D(k,n)\\
&\sim O\left(e^{c\sqrt{m-\lfloor m^{{9}/{16}}\rfloor}}\right)+\sum_{m-\lfloor m^{{9}/{16}}\rfloor<k\le m}\frac{5c}{2^5\cdot 3}
\frac{e^{c\sqrt{k}}}{k^2}\left(1+e^{-\frac{c(n-k)}{2\sqrt{k}}}\right)^{-2}\\
&\sim O\left(e^{c\sqrt{m}-\frac{c}{2}m^{{1}/{16}}}\right)+\frac{5ce^{c\sqrt{m}}}{2^5\cdot 3m^2}
\sum_{0\le k<\lfloor m^{{9}/{16}}\rfloor}e^{-\frac{ck}{2\sqrt{m}}+O(m^{-3/8})}\left(1+e^{-\frac{c(n-m+k)}{2\sqrt{m-k}}}\right)^{-2},
\end{align*}
that is,
\begin{align}\label{eq20}
\pi(m,n)&\sim \frac{5ce^{c\sqrt{m}}}{2^5\cdot 3m^2}
\sum_{0\le k<\lfloor m^{{9}/{16}}\rfloor}e^{-\frac{ck}{2\sqrt{m}}}\left(1+{\bf 1}_{n-m<m^{9/16}}e^{-\frac{c(n-m+k)}{2\sqrt{m-k}}}\right)^{-2}\nonumber\\
&=\frac{5ce^{c\sqrt{m}}}{2^5\cdot 3m^2}
\sum_{0\le k<\lfloor m^{{9}/{16}}\rfloor}e^{-\frac{ck}{2\sqrt{m}}}\left(1+{\bf 1}_{n-m<m^{9/16}}e^{-\frac{c(n-m+k)}{2\sqrt{m}}+O(m^{-3/8})}\right)^{-2}\nonumber\\
&\sim \frac{5ce^{c\sqrt{m}}}{2^5\cdot 3m^2}
\sum_{0\le k<\lfloor m^{{9}/{16}}\rfloor}e^{-\frac{ck}{2\sqrt{m}}}\left(1+e^{-\frac{c(n-m+k)}{2\sqrt{m}}}\right)^{-2}.
\end{align}
On the other hand, by using Abel's summation formula it is easy to find that
\begin{align*}
\sum_{0\le k<\lfloor m^{{9}/{16}}\rfloor}\frac{e^{-\frac{ck}{2\sqrt{m}}}}{\left(1+e^{-\frac{c(n-m+k)}{2\sqrt{m}}}\right)^{2}}\sim
\int_{0}^{\infty}\frac{e^{-\frac{cx}{2\sqrt{m}}}}{\left(1+e^{-\frac{c(n-m+x)}{2\sqrt{m}}}\right)^{2}}\,dx=\frac{2\sqrt{m}}{c}\frac{1}{1+e^{-\frac{c(n-m)}{2\sqrt{m}}}}.
\end{align*}
Therefore by combining \eqref{eq20} and above, if $m\le n$ and $m\rrw +\infty$ then
$$\pi(m,n)\sim \frac{5}{2^4\cdot 3}\frac{e^{c\sqrt{m}}}{m^{3/2}}\left(1+e^{-\frac{c(n-m)}{2\sqrt{m}}}\right)^{-1}.$$
Finally using \eqref{eq11} then the proof of Theorem \ref{th2} follows.





\bigskip
\noindent
{\sc School of Mathematical Sciences\\
 East China Normal University\\
500 Dongchuan Road\\
Shanghai 200241\\
PR China}\newline
\href{mailto:nianhongzhou@outlook.com}{\small nianhongzhou@outlook.com}

\end{document}